\documentclass[a4paper,12pt]{article}

\input xy
\xyoption{all}
\usepackage{amsbsy,eucal,amssymb,amsmath}
\usepackage{latexsym,amsfonts,amsthm}

\theoremstyle{plain}
\newtheorem{theorem}{Theorem}[section]
\newtheorem{lemma}[theorem]{Lemma}

\newtheorem{proposition}[theorem]{Proposition}

\theoremstyle{definition}
\newtheorem{definition}[theorem]{Definition}
\newtheorem{example}[theorem]{Example}

\theoremstyle{remark}

\newcommand{\set}[1]{\{#1\}}

\newcommand{\RR}{\mathbb{R}}
\newcommand{\ba}{{\bf a}}
\newcommand{\bb}{{\bf b}}
\newcommand{\bB}{{\bf B}}
\newcommand{\bBr}{{\bf B_r}}
\newcommand{\bT}{{\bf T}}
\newcommand{\bTr}{{\bf T_r}}
\newcommand{\bR}{{\bf R}}
\newcommand{\bM}{{\bf M}}
\newcommand{\bMr}{{\bf M_r}}
\newcommand{\bc}{{\bf c}}
\newcommand{\bpi}{{\boldsymbol\pi}}
\newcommand{\mnu}{^{-1}}
\providecommand{\keywords}[1]{\textbf{\textit{Keywords:}} #1}

\begin{document}

	\title{Structural Relations of Symmetry among Players in Strategic Games}
	
	\author{Fernando Tohm\'e and Ignacio Viglizzo}
	
\maketitle
	
	\begin{abstract}
		The notions of symmetry and anonymity in strategic games have been formalized in different ways in the literature. We propose a combinatorial framework to analyze these notions, using group actions. Then, the same framework is used to define partial symmetries in payoff matrices. With this purpose, we introduce  the notion of the role a player plays with respect to another one, and combinatorial relations between roles are studied. Building on them, we define relations directly between players, which provide yet another characterization of structural symmetries in the payoff matrices of strategic games. 
	\end{abstract}
	

	\section{Introduction}
	
	Game Theory conceives complete information strategic games as interactive decision problems in which the decision-makers know all the relevant parameters. A payoff matrix is a representation of such a problem.  The literature studies what a solution is and determines conditions for its existence. But payoff matrices can be studied as purely combinatorial objects reflecting the systems of relations involved in the formulation of their underlying interactive decision problems. 
	
	This is precisely the goal of this work, namely to investigate some new properties of the payoff matrices of strategic games. The main focus of this inquiry are the different kinds of symmetries {\it among players} in such matrices. The point of departure of our analysis is the original contribution of John Nash \cite{nash51noncooperative} in which he presented a formal definition of {\em symmetric games} using permutations over the set of all actions.
	
	The literature extended Nash's definition. So, for instance, Dasgupta and Maskin provided their own characterization \cite{dasgupta86existence}. They were motivated by the intuition that a game can be seen as symmetric if payoffs are invariant under permutations of the identities of the players. More recently, other notions of symmetry in games have been studied in \cite{stein11exchangeable} and \cite{ham18symmetry}.

	Interestingly, such notion of symmetry in a game is quite analogous to the concept of {\em anonymity} in Social Choice procedures. Both leave invariant some elements in the representation (payoffs in games, choices in the case of social choice functions) under some permutations of the names of the individual agents. The group-theoretic foundations of anonymity are in many ways related to the presentation in this paper, although they lead to results of a different style (see for instance \cite{kelly92abelian}, \cite{chichilnisky96actions} or \cite{serizawa99strategy}). Further definitions of symmetry and anonymity in games have been introduced in the literature on Algorithmic Game Theory (\cite{daskalakis2007computing}, \cite{brandt08symmetries}). 
	
	In this article, we start by proving precise connections between some of the aforementioned definitions of symmetry and anonymity in payoff matrices. But the central point of our study concerns the characterization of some {\em partial} symmetries in games. To motivate this idea consider a game involving three players in which player $1$ can exchange places with player $2$ without a change in payoffs, while an exchange with player $3$ leads to changes of the payoffs. Thus, there is some sort of symmetry between players $1$ and $2$  that does not hold between players $1$ and $3$. This example points towards a new notion, that has not yet been treated before, namely the {\em role} that a player $i$ may play with respect to another one, $j$.  By this we mean how a change in $i$'s choice of actions modifies $j$'s payoffs independently of how the rest of the players is affected.  We define different relations comparing roles  by means of sets of equations. Depending on which set of equations we choose, we obtain three different binary relations among roles, which we name {\em blind}, {\em twisted} and {\em simulation}.
	
	The blind, twisted and simulation relations defined for roles can be used as a tool to define relations between {\em players}. Thus, we characterize whether a pair of players can be seen as interchangeable according to the payoff matrix of a strategic game, or whether a player can see her situation reflected in the possible actions and outcomes of other players. The notion of role becomes relevant in capturing the different kinds of views a player may entertain. 
	
	This paper is only concerned with symmetry-related structural properties of payoff matrices in strategic games. The notion of roles and the relations among them can be seen as a contribution to the system-theoretical conception of games, according to which a game can be seen as a {\em system of interrelated objects}. Thus, it is of interest to know different types of relation holding among its components.
	
	\subsection{Plan of the paper}
	
	Even if symmetries can be defined among the labels of strategies or among different games, we restrict our attention here on the {\em symmetry among players}. To start the analysis, in Section 2 we introduce and compare different definitions of symmetry (and the related concept of anonymity) present in the literature on Game Theory. We view these symmetries in terms of the action of the group of permutations $S_I$ (among players) on the set of all strategy profiles and show the equivalences between some of these definitions.   
	
	In Section 3 we present an alternative characterization of symmetry in  games in terms of invariant properties under permutation. While this presentation is close to other contributions in the literature, it provides a framework for the analysis of other kinds of symmetry. Before running this study we examine the equivalence between our characterization and some of the concepts introduced in Section 2.
	
	Finally, in Section 4 we present the core of our analysis, by introducing the concept of {\em roles} of player. The idea is to isolate the effects of the actions of a single player on the payoffs of another. We present and compare three different types of roles, {\em blind}, {\em twisted} and {\em simulated}. Then, in Section 5 we change the focus of analysis from the relations between {\em roles} to those between {\em players}. In this new setting we detect two additional concepts of partial symmetry among players. 
	
	Finally, Section 6 discusses briefly the meaning of these results.

	\section{Notions of symmetry and anonymity in the literature}

	In this section we will review different notions of symmetry and the related concept of anonymity presented in the literature on games. We prove equivalences among some of them.  We start by introducing some preliminary  definitions:
	
	\begin{definition}\label{definitionGame}
		Let $G = \langle I, \{A_i\}_{i \in I}, \{\pi_i\}_{i \in I} \rangle$ be a {\em strategic game}, where $I=\set{1,\ldots, n}$ is a set of {\em players} and $A_i, i\in I$ is a finite set of {\em strategies} for each player. A {\em strategy profile}, $\ba=(a_1, \ldots, a_n)$ is an element of $\prod_{i \in I} A_i$.
		In turn, $\pi_i : \prod_{i \in I} A_i \rightarrow \mathbb{R} $ is player $i$'s payoff.
	\end{definition}
	
	If we let $A=\bigcup_{i\in I}A_i$, we can assume that all the players choose an action from the same set $A$ so the strategy profiles are elements of $A^I$. We denote with $\ba=(a_1,\ldots,a_n)$ an element of $A^I$ and with $\ba_{-i}$ the vector $(a_1,\ldots,a_{i-1},a_{i+1},\ldots, a_n)$. Nevertheless, we still use $A_i$ to denote the $i$-th coordinate of $A^I$.
	
	\begin{definition}
		$S_n$ is the group of permutations over a set of $n$ elements. 
		For $I=\set{1,2,\ldots,n}$, we may also denote this group with  $S_I$. A \emph{right group} action of a group $\mathcal{G}$ on an arbitrary set $X$ is an operation $X \times \mathcal{G} \to X$ (denoted as a product) satisfying two axioms: $xe = x$ for all $x \in X$, where $e$ the identity element of $\mathcal{G}$, and $x(gh) = (xg)h$ for all $g, h \in \mathcal{G}$ and all $x \in X$. Similarly, a left action of a group on a set can be defined.
	\end{definition}
	
	\begin{definition}\label{symmetricGame}
		\cite{dasgupta86existence}
		A game $G$  is {\em symmetric} if for any permutation $\sigma \in S_I$  and every profile of strategies $(a_1, \ldots, a_n)$ we have that for all $i\in I$,
		
		\begin{equation}\label{eqDM}
		\pi_i(a_1,\ldots, a_n)= \pi_{\sigma(i)}(a_{\sigma(1)},\ldots, a_{\sigma(n)}).\end{equation}
	\end{definition}
	
	Since permutations are an important tool in this work, we analyze how they act on profiles\footnote{Similar, equivalent notation and results can be found in \cite{stein11exchangeable} and \cite{ham18symmetry}.}:
	
	\begin{lemma}
		Let $\ba=(a_1,\ldots,a_n)\in A^I$. The operation defined by
		\[\ba\sigma=(a_1,\ldots,a_n)\sigma=(a_{\sigma(1)},\ldots,a_{\sigma(n)})\]
		is a right action of the group $S_n$ over the set $A^I$ of all profiles.
	\end{lemma}
	\begin{proof}
		Let $\sigma$ and $\tau$ be two permutations. Then for all $\ba\in A^I$
		\[(\ba\sigma)\tau=\ba(\sigma\tau).\]
		To see this we calculate
		\[(\ba\sigma)\tau=(a_{\sigma(1)},\ldots,a_{\sigma(n)})\tau\]
		We rename now $\ba\sigma={\bf b}$, that is, for $i=1,\ldots,n$, $b_i=a_{\sigma(i)}$. But then ${ b}_{\tau(i)}=a_{\sigma\tau(i)}$ so $(\ba \sigma)\tau={\bf b}\tau=(b_{\tau(1)},\ldots,b_{\tau(n)})=\ba(\sigma\tau)$.
	\end{proof}
	
	Using this notation, Definition \ref{symmetricGame} may be restated saying that for all $i\in I, {\bf a}\in A^I$ and $\sigma \in S_n$,\footnote{Nash \cite{nash51noncooperative} also defines symmetric games in terms of permutations of actions, which in turn lead to permutations of the name of players. Depending on how one interprets his notation, this may lead either to Definition \ref{symmetricGame} or to our characterization (definition \ref{permutationInvariant}) below.} 
	
	\begin{equation}\label{eqDMredux}
	\pi_i({\bf a})=\pi_{\sigma(i)}(\ba\sigma).
	\end{equation}
	
	\begin{example}\label{Parikh} Consider a game $G$ in which $I = \{1,2,3,4,5\}$ and $A = \{s,r,t\}$. Consider a profile ${\bf a} = (a_1, a_2, a_3, a_4, a_5)$$=$$(s,r,t,r,r)$ and a permutation $\sigma = (12)$ (i.e. it exchanges the names of players $1$ and $2$). Then, for instance, 
		$$\pi_1(s,r,t,r,r) = \pi_1(a_1, a_2, a_3, a_4, a_5) = \pi_{\sigma(1)}(b_1,b_2,b_3,b_4,b_5) = $$
		$$ = \pi_2 (a_{\sigma(1)}, a_{\sigma(2)}, a_{\sigma(3)}, a_{\sigma(4)}, a_{\sigma(5)}) = \pi_2(r,s,t,r,r).$$
	\end{example}
	
	\begin{definition} \cite{parikh66context}
		The {\em commutative image} of an action profile $\ba \in A^I $ is given by $\#\ba = (\#(a, \ba))_{ a\in A}$ where $\#(a, \ba) = |\{i \in I : a_i = a\}|$. In other words, $\#(a, \ba)$ denotes the number of players playing
		action $a$ in the profile $\ba$, and $\#\ba$ is the vector of these numbers for all the different actions. 
	\end{definition}
	
	\begin{example}\label{Parikh1}
		Consider again, as in Example~\ref{Parikh}, ${\bf a} = (s,r,t,r,r)$. Then $\#(s, (s,r,t,r,r)) = 1$, $\#(r, (s,r,t,r,r)) = 3$ and $\#(t, (s,r,t,r,r)) = 1$. Thus, $\#(s,r,t,r,r) = (1,3,1)$.
	\end{example}
	
	\begin{lemma} \label{sigmaexists}
		For any $\ba,\bb\in A^I$, $\#\ba=\#\bb$ if and only if there exists a permutation $\sigma$ of $I$ such that $\ba=\bb\sigma$.
	\end{lemma}
	
	\begin{proof}
		Let $\ba=(a_1,a_2,\ldots,a_n)$. If we assume that $\#\ba=\#\bb$, then $\#(a_1,\ba)=\#(a_1,\bb)\ge 1$. We can define $\sigma(1)$ to be the first index $i$ such that $b_i=a_1$. Proceeding in this fashion, one may construct the required permutation $\sigma$.
		
		If  $\mathbf{a} = \mathbf{b}\sigma$, then by definition, for every $i$,  $b_{\sigma(i)} = a_i$, so for every ${a} \in A$,  $\{ i: b_{\sigma(i)} = {a}\} = \{ i: a_{i} = {a}\}$. Therefore,  $\#(a,\bb\sigma)=|\{ i: b_{\sigma(i)} = {a}\} | = \#(a,\ba)=| \{ i: a_{i} = {a}\}|$ and $\#(\mathbf{b}\sigma) = \#(\mathbf{a})$.
	\end{proof}

	\begin{definition} \cite{daskalakis2007computing} \label{defDP}        An anonymous game $G = \langle I,A,\{u^i_k\}\rangle$ consists of
		a set $I = \set{1,...,n} $ of $n \ge 2$ of players, a set $A$ of $s \ge 2$ actions, and a set of $ns$ utility functions, where $u^i_a$, with    
		$i \in I$ and  $a\in A$  is the utility of player $i$ when she plays action $a$, a function  mapping the set of partitions 
		$P_{n-1}=   \set{(x_a)_{a\in A}: x_a \in \mathbb{N} \mbox{ for all } a \in A,	\sum_{a\in A}  x_a  = n - 1}=\set{\#\ba_{-i}:\ba\in A^I}$ to  $\RR$. 
	\end{definition}
	
	\begin{example}\label{Parikh2}
		Consider again as in Example~\ref{Parikh}, ${\bf a} = (s,r,t,r,r)$. Since $n=5$, $u^3_t : P_{5-1} \rightarrow \RR$. Here $P_4 = \{(4,0,0), (3,1,0), \ldots, (0,1,3), (0,0,4)\}$.
	\end{example}
	
	To see how the utility functions $u^i_a$ are related to  the payoff functions introduced above, we denote with $\cdot_i:A^I\to A$ the $i$-th projection, and  we consider the map that assigns to each $\ba\in A^I$ the partition $\#\ba_{-i}$. Let $u^i_\cdot:A\times P_{n-1}\to \RR$ be the function\footnote{This is equivalent (via curryfication) to saying that $u^i_\cdot:A\to \RR^{P_{n-1}}$, given that $u^i_a:P_{n-1}\to\RR$.} that assigns to each pair $(a,\#\ba_{-i})$ the value $u^i_a(\#\ba_{-i})$. Under these conditions, the following diagram commutes:

	\[\xymatrix{A^I\ar[r]^{\pi_i}\ar[d]_{\langle\cdot_i,\#_{-i}\rangle}&\RR  \\
		A\times P_{n-1}\ar[ru]_{u^i_\cdot}&}  \]

	The corresponding equation is 
	\begin{equation}\label{piu}
	\pi_i(\ba)=u^i_{a_i}(\#\ba_{-i}),
	\end{equation}
	which lets us obtain the functions $u^i_a$ given the $\pi_i$, but this does not work in the other direction. The utility functions have less information, since we have removed the identities of the players playing each of the actions, keeping only the number of players choosing each of the actions.             
	
	\begin{example}
		Consider the case described in Example~\ref{Parikh2}:  $\langle \cdot_3, \#_{-3} \rangle$ applied on $(s,r,t,r,r)$ yields $(t, (1,3,0))$, since ${\bf a}_{-3} = (s,r, r,r)$ and $\#_{-3}\ba=(1,3,0)$. Applying $u^3_{\cdot}$ on $(t, (1,3,0))$ gives $u^3_t(1,3,0)$. The commutativity of the diagram indicates then, that $\pi_3(s,r,t,r,r) = u^3_t(1,3,0)$.
	\end{example}

	In \cite{brandt08symmetries}, the following definitions are given:
	
	\begin{definition} \label{defBetal}
		Let $G = (I,A,\set{\pi_i}_{i\in I} )$ be a game. $G$ is called
		\begin{itemize}
			\item \textit{anonymous} when  for all $i \in I$ and all $\ba,\bb \in A^I$, if $a_i =b_i$ and $\#\ba_{-i}  = \#\bb_{-i}$, then $\pi_i (\ba) = \pi_i (\bb)$. 
			\item \textit{symmetric} when  for all $i,j \in I$ and all $\ba,\bb \in A^I$, if $a_i =b_j$ and $\#\ba_{-i}  = \#\bb_{-j}$, then $\pi_i (\ba) = \pi_j (\bb)$. 
			\item \textit{self-anonymous} when  for all $i \in I$ and all $\ba,\bb \in A^I$, if  $\#\ba = \#\bb$, then $\pi_i (\ba) = \pi_i (\bb)$. 
			\item \textit{self-symmetric} when  for all $i,j \in I$ and all $\ba,\bb \in A^I$, if  $\#\ba = \#\bb$, then $\pi_i (\ba) = \pi_j (\bb)$. 
		\end{itemize}
	\end{definition}
	
	The two previous definitions agree on the meaning of anonymity in games:
	\begin{proposition}
		The definitions \ref{defDP} and \ref{defBetal} of an anonymous game are equivalent.
	\end{proposition}
	
	\begin{proof}
		Assume that a game $G$ is anonymous according to definition \ref{defDP}. Let $i\in I$ be fixed,  and consider profiles $\ba$ and $\bb$ such that $a_i=b_i$ and $\#\ba_{-i}=\#\bb_{-i}$. Then, using equation (\ref{piu}) $\pi_i(\ba)=u^i_{a_i}(\#\ba_{-i})=u^i_{b_i}(\#\bb_{-i})=\pi_i(\bb)$.
		
		Conversely, the conditions from definition \ref{defBetal}, ensure that equation \eqref{piu} defines the functions $u^i_a$ on $P_{n-1}$.
	\end{proof}

	In the case of two players, Definition \ref{symmetricGame} captures the intuition about what a symmetric game should be, such as the Prisoner's Dilemma. However, for more than 3 players, this definition sets too many restrictions, resulting in games in which for all permutations of a given strategy profile, all players get the same payoff.  After proving this, we also prove that  Definition \ref{symmetricGame} is equivalent to the one of self-symmetric games from Definition \ref{defBetal}.
	
	\begin{lemma}\label{overdet}
		If $n\ge 3$, then in symmetric games  in the sense of Definition \ref{symmetricGame},  for all profiles $\ba\in A^I$, and all $i,j\in I$,
		\[\pi_i(\ba)=\pi_j(\ba).\]
	\end{lemma}
	\begin{proof}
		Consider a profile of strategies $\ba$. Choose some three different indexes $i,j,k$. Under the permutation $(i\ k)$ we have by equation \eqref{eqDMredux} that $\pi_i(\ba)$ $=$ $\pi_k(\ba (i\ k))$. In turn, under the cycle $(i\ j)$ we have that  \[ \pi_k(\ba(i\ k))=\pi_k(\ba(i\ k)(i\ j)).\]
		Notice that $(i\ k)(i\ j)=(i\ j\ k)$.  Finally, under the the inverse permutation $(i\ k\ j)$ we have that \[\pi_k(\ba(i\ j\ k))=\pi_j(\ba(i\ j\ k)(i\ k\ j))=\pi_j(\ba).\]  Thus,  $\pi_i(\ba)=\pi_j(\ba)$. 
	\end{proof}

	\begin{proposition} \label{ssS}
		A game is self-symmetric if and only if it is symmetric as in Definition \ref{symmetricGame}.
	\end{proposition}
	\begin{proof}
		Assume that we have a self-symmetric game and take $\ba\in A^I$, $i\in I$ and $\sigma\in S_I$. Letting $j$ be $\sigma(i)$ and $\bb=\ba\sigma$, we have by Lemma \ref{sigmaexists} that $\#\ba=\#\bb$, so $\pi_i(\ba)=\pi_j(\bb)=\pi_{\sigma(i)}(\ba\sigma)$.
		
		Now assume that a game is symmetric according to Definition \ref{symmetricGame} and consider $i, j \in I$, $\ba, \bb\in A^I$ with $\#\ba=\#\bb$. By Lemma \ref{sigmaexists}, there exists $\sigma\in S_I$ such that $\bb=\ba\sigma$. Then $\pi_i(\ba)=\pi_{\sigma(i)}(\ba\sigma)=\pi_{\sigma(i)}(\bb)$, and by Lemma \ref{overdet}, this is equal to $\pi_j(\bb)$.

	\end{proof}

	\section{Symmetry as Permutation invariance}

	A  notion of symmetry which is better suited to our ends can be obtained by considering that the actions of the group of permutations over payoffs and profiles commute. To see this, first notice that we may consider the vector of functions $\bpi=(\pi_1,\pi_2,\ldots,\pi_n):A^I\to \RR^n$. The group $S_I$ acts on this vector as well, defining, as we did for action profiles, $\bpi\sigma=(\pi_{\sigma(1)},\pi_{\sigma(2)},\ldots,\pi_{\sigma(n)})$.
	
	\begin{definition}
		\label{permutationInvariant}
		A game $G$  is {\em invariant with respect to a permutation} $\sigma$ if for every strategy profile  $\ba\in A^I$ we have that 
		
		\[\bpi\sigma(\ba)=\bpi(\ba\sigma).\]
		
		\[\xymatrix{A^I\ar[r]^\bpi\ar[d]_\sigma&\RR^n\ar[d]^{\sigma}\\A^I\ar[r]_\bpi&\RR^n}\]
		
	\end{definition}
	
	This means that for all $i\in I$,
	
	\[\pi_{\sigma(i)}(a_1,\ldots,a_n)=\pi_i(a_{\sigma(1)},\ldots,a_{\sigma(n)}).  \]

	Notice that in the right hand side of this equation, player $i$ is playing the action $a_{\sigma(i)}$, the same one that player $\sigma(i)$ is playing in the right hand side, and they both get the same payoff. 
	
	An equivalent equation is:\footnote{In \cite{Bouyer17Nash} this characterization of symmetry is presented as a fix to Definition~\ref{symmetricGame}. \cite{Hefti17equilibria} uses an alternative characterization: $\pi_{i}(a_1,\ldots,a_n)=\pi_{\sigma(i)}(a_{\sigma^{-1}(1)},\ldots,a_{\sigma^{-1}(n)})$.}
	
	\begin{equation}\label{symmetric}
	\pi_i(a_1,\ldots, a_n)= \pi_{\sigma^{-1}(i)}(a_{\sigma(1)},\ldots, a_{\sigma(n)})\end{equation}

	One may also consider a game that is invariant under the action of some permutations:
	
	\begin{definition}
		A game $G$ is invariant under the set of permutations $X\subseteq S_I$ if it is invariant with respect to each permutation $\sigma\in X$.
	\end{definition}
	
	The group action over the set of profiles  determines a partition of $A^I$ in  \textit{orbits}: the orbit of a strategy profile $\ba$ is the set 
	\[O(\ba)=\set{{\bf b}\in A^I:{\bf b}=\ba\sigma \mbox{ for some }\sigma\in S_n}.\]
	The orbits  coincide with the sets of profiles with the same commutative image, this is: $O(\ba)=\set{\bb\in A^I:\#\bb=\#\ba}$.

	\begin{example}\label{exampleOver}
		Consider a game with three players, with the following payoff structure, in which each player has two actions to choose from, $a,b$ (player $1$ chooses rows, $2$ columns and $3$ matrices):
		\begin{center}
			$a$ \ \
			\begin{tabular}{l|l|l}
				
				&  $a$ & $b$ \\
				\hline
				$a$ & (10,10,10) & (5,5,5) \\
				\hline  $b$ & (5,5,5) & (-5,-5,-5)
			\end{tabular}
			\ \ \ \
			$b$\ \
			\begin{tabular}{c|l|l}
				& $a$ & $b$ \\
				\hline
				$a$ & (5,5,5) & (-5,-5,-5) \\   \hline
				$b$ & (-5,-5,-5) & (0,0,0)
			\end{tabular}
		\end{center}
		According to Lemma  \ref{overdet} and Proposition \ref{ssS}, this structure of payoffs under $S_3$ is self-symmetric. The orbits are $ \{(a,a,a)\}$, $\{(a,b,a), (b,a,a), (a,a,b)\}$, $\{(b,b,a), (b,a,b)$, $(a,b,b)\}$, and $\{(b,b,b)\}$. Here all the players get the same payoff, in all the profiles of the same orbit.
		
	\end{example}

	\begin{lemma}\label{composition}
		Let $G$ be a game and $\tau, \sigma$ two permutations under which the game is invariant. Then the game is invariant under their composition.
	\end{lemma}
	\begin{proof}
		By hypothesis, for all profiles $\ba, \bpi\sigma(\ba)=\bpi(\ba\sigma)$. This means that for every $j$, $\pi_{\sigma(j)}(\ba)=\pi_j(\ba\sigma)$. In particular, this is true for $j=\tau(i)$ so for all $i$, $\pi_{\sigma\tau(i)}(\ba)=\pi_{\tau(i)}(\ba\sigma)$.  Now we can write $\bpi\sigma\tau(\ba)=\bpi\tau(\ba\sigma)=\bpi(\ba\sigma\tau)$.
	\end{proof}
	
	\begin{theorem}
		The set of permutations under which  a game is invariant forms a group.
	\end{theorem}
	\begin{proof}
		It is clear that every game is invariant under the identity permutation $()$. Lemma \ref{composition} proves that the set is closed under the group multiplication. Finally, if $G$ is invariant under $\sigma$, since $S_n$ is a finite group, there is some $k$ such that $\sigma^k=()$, so $\sigma^{-1}=\sigma^{k-1}$. Using Lemma \ref{composition}, we see that $\sigma\mnu$ is in the set as well.
	\end{proof}
	
	Finally we propose a definition of symmetry as invariance under permutations and prove it equivalent to the one in Definition \ref{defBetal}.
	
	\begin{definition} \label{defSymPer}
		A game $G$ is symmetric if it is invariant under all the permutations in $S_I$.
	\end{definition}
	
	\begin{theorem}
		The definitions of symmetric game in \ref{defBetal} and \ref{defSymPer} are equivalent.
	\end{theorem}

	\begin{proof}
		Consider a fixed permutation $\sigma$, a profile $\ba$ and $i\in I$. Letting $\bb=\ba\sigma$, and $j=\sigma^{-1}(i)$, we have that $b_j=a_{\sigma(j)}=a_i$ and therefore $\#\ba_{-i}= \#\bb_{-j}$, so by Definition \ref{defBetal}, if $G$ is symmetric, $\pi_i(\ba)=\pi_j(\bb)=\pi_{\sigma^{-1}(i)}(\ba\sigma)$.
		
		Next we choose profiles $\ba$ and $\bb$ such that $a_i=b_j$ and $\#\ba_{-i}= \#\bb_{-j}$. Then we define $\sigma(j)=i$ and complete the permutation so that $\bb=\ba\sigma$. Then $\pi_j(\bb)=\pi_j(\ba\sigma)=\pi_{\sigma(j)}(\ba)=\pi_i(\ba)$.
	\end{proof}
	
	\section{Roles players play}
	
	An advantage of defining symmetry in games in terms of permutations is that we can relax some of the conditions to find more general, but still useful, definitions. While its characterization based on commutative images involves a loss of information, the use of permutations allows us to keep track of which player is playing each of the actions.
	
	We start by  isolating the effects  a player's action has over the payoffs of another. The payoffs for a player $i$ are given by the function $\pi_i:A^I\to\RR$. We  decompose $A^I$ as $A_i\times A_j\times A_{-ij}$, where $A_{-ij}=\prod_{k\notin \set{i,j}} A_k$. There is a bijective correspondence between the functions from $A_i\times A_j\times A_{-ij}$ to $\RR$ and the functions from $A_{-ij}$ to the set of real-valued functions on $A_i\times A_j$. We use this correspondence to define the following notion:

	\begin{definition}
		Given two different players,  $i, j$, the {\em role} that player $j$ plays for player $i$ is the function  $r_i^j:A_{-ij}\to\RR^{A_i\times A_j}$ such that for all $\ba=(a_1,\ldots,a_i,\ldots,a_j,\ldots,a_n)$,
		\[[r_i^j(\ba_{-ij})](a_i,a_j)=\pi_i(\ba),\]

		where $\ba_{-ij}$ is the projection of $\ba$ over $\displaystyle A_{-ij}$.
		
		If $i=j$, we define   $r_i^i:A_{-i}\to \RR^{A_i}$ to be such that   for all $\ba\in A^I$, \[[r_i^i(\ba_{-i})](a_i)=\pi_i(\ba).\]
	\end{definition}
	If we want to determine whether the role that player $j$ plays for player $i$ is somehow equivalent to the one that player $l$ plays for player $k$, we need to establish a correspondence between $A_{-ij}$ and $A_{-kl}$.  Any such correspondence is a permutation $\sigma$ of $I$ such that $\sigma(k)=i$ and $\sigma(l)=j$, but there is not a canonical way of choosing one. Once such a permutation is chosen, we can say that the roles are the same under this permutation if for every profile $\ba$ the payoff for player $i$ when they play action $a_i$ and player $j$ plays $a_j$ is the same as the payoff for player $k$ when they play action $a_i$ and player $l$ plays action $a_j$. The equation for this is: 
	\begin{multline} \label{pieqpi} \pi_i(a_1,\ldots, a_i,\ldots,a_j,\ldots,a_k,\ldots,a_l,\ldots,a_n)=\\
	\pi_k(a_{\sigma(1)},\ldots, a_{\sigma(i)},\ldots,a_{\sigma(j)},\ldots,a_i,\ldots,a_j,\ldots,a_{\sigma(n)})
	\end{multline}
	
	We can write equation (\ref{pieqpi}) as:
	\begin{equation}\label{pieqpi2}
	\pi_i(a_i,a_j,\ba_{-ij})=\pi_k(a_i,a_j,\ba_{-ij}\sigma), 
	\end{equation}
	
	or, more compactly:
	\begin{equation}\label{pieqpi3}
	\pi_i(\ba)=\pi_k(\ba\sigma). 
	\end{equation}

	Since this holds for all $\ba\in A^I$, we can abstract away the variables $a_i$ and $a_j$ to express this relation between roles:
	
	\begin{definition}
		A \textit{role $r_i^j$ is the same as the role $r_k^l$ under a permutation $\sigma$} such that $\sigma(k)=i$ and $\sigma(l)=j$ if for every profile $\ba\in A^I$,
		\begin{equation}\label{roeqro}
		r_i^j(\ba_{-ij})=r_k^l(\ba_{-ij}\sigma).
		\end{equation}
		
		For the case in which $i=j$, we write
		\begin{equation}\label{roeqroik}
		r_i^i(\ba_{-i})=r_k^k(\ba_{-i}\sigma),
		\end{equation}
		where $\sigma$ is any permutation such that $\sigma(k)=i$.
	\end{definition}

	Different sets of permutations in the conditions above yield diverse sets of equations, and therefore, different definitions of binary relations over the set of roles in the game. First, we examine the case where we ask the condition to hold for every acceptable permutation:
	\begin{definition}\label{blindr}
		$r_i^j{\bf B_r}r_k^l$: the roles $r_i^j$ and $r_k^l$ are \textit{blindly related} if \eqref{roeqro} holds for all $\ba$ and $\sigma$ such that $\sigma(k)=i$ and $\sigma(l)=j$.
	\end{definition}
	
	This relation is not reflexive if there are at least three players: it is easy to give an example of a game in which three players can choose action $a, b$ or $c$ and $\pi_1((a,b,c))\neq\pi_1((a,b,c)(23))=\pi_1((a,c,b))$, so $r_1^1\bBr r_1^1$ does not hold.
	
	\begin{lemma}\label{blindST}
		The relation ${\bf B_r}$ is symmetric and transitive.
	\end{lemma}
	
	\begin{proof}
		Suppose that for all $\ba\in A^I$ and $\sigma$ such that $\sigma(k)=i$ and $\sigma(l)=j$, $\pi_i(\ba)=\pi_k(\ba\sigma)$. Let $\bb\in A$ and $\tau$ be a permutation such that $\tau(i)=k$ and $\tau(j)=l$. Then $\tau^{-1}$ satisfies that $\tau^{-1}(k)=i$ and $\tau^{-1}(l)=j$ so by hypothesis, taking $\ba=\bb\tau$ and $\sigma=\tau^{-1}$,  $\pi_k(\bb)=\pi_k(\bb\tau\tau^{-1})=\pi_i(\bb\tau)$.
		
		For transitivity, assume that $r_i^j{\bf B_r}r_k^l$ and $r_k^l{\bf B_r}r_s^t$. Then we have that for any $\ba\in A^I$ and $\sigma$ such that $\sigma(k)=i$ and $\sigma(l)=j$, $\pi_i(\ba)=\pi_k(a\sigma)$, and for any $\bb\in A^I$ and $\tau$ such that $\tau(s)=k$ and $\tau(t)=l$, $\pi_k(\bb)=\pi_s(\bb\tau)$. 
		
		Now consider $\bc\in A^I$ and let $\gamma$ be a permutation such that $\gamma(s)=i$ and $\gamma(t)=j$. We must factor $\gamma$ as the composition of two permutations, $\sigma$ and $\tau$ satisfying the conditions above, in order to be able to use the hypothesis. 
		
		If the players $k, l, s$ and $t$ are different, or $k=s$, or $t=l$, or both, we let $\sigma=\gamma(k\ s)(l\ t)$ and $\tau=(l\ t)(k\ s)$, then $\sigma$ and  $\tau$ are as in the hypothesis, with $\gamma=\sigma\tau$. It follows that $\pi_i(\bc)=\pi_k(\bc\sigma)=\pi_s(\bc\sigma\tau)=\pi_s(\bc\gamma)$, so $r_i^j{\bf B_r}r_s^t$.
		
		Notice that if one of the equalities $i=j$, $k=l$ or $s=t$ hold, then the other two hold as well. In this case it is enough to let $\sigma=\gamma(ks)$ and $\tau=(ks)$.
		
		If 	$t=k$ and $l=s$ let $\sigma=\gamma(lk)$ and $\tau=(lk)$, while if $t=k$ but $l\neq s$, we should take $\sigma=\gamma(lks)$ and $\tau=(lsk)$.
		
		Finally, if $t\neq k$ and $l=s$ let $\sigma=\gamma(klt)$ and $\tau=(ktl)$.

	\end{proof}

	It follows from the previous Lemma that if there is a role that is blindly related to a role $r_i^j$, then $r_i^j$ is blindly related to itself. If we recall that for every $\ba_{-ij}\in A_{-ij}$, $r_i^j(\ba_{-ij})$ is a function from $A_i\times A_j\to\RR$, we have that for every permutation $\sigma$ leaving $i$ and $j$ unchanged, and every $\ba_{-ij}$, the functions $r_i^j(\ba_{-ij})$and $r_i^j(\ba_{-ij}\sigma)$ coincide\footnote{Alternatively, we can say that the role $r_i^j$ is invariant under the subgroup of permutations that leave $i$ and $j$ fixed.}.  This relation holds when player $i$ is indifferent (or `blind' with respect) to who are the players choosing all the actions in a profile except for those from players  $i$ and $j$ (a condition akin to anonymity), and clearly this is not always the case. If  $i=j$, this indifference extends to all players other than $i$. 
	
	Furthermore, if we have that $r_i^j{\bf B_r}r_k^l$, then both roles share this characteristic and their functions $r_i^j(\ba_{-ij})$ and $r_k^l(\ba_{-kl}\sigma)$ agree for all $\sigma$ such that $\sigma(k)=j$ and $\sigma(l)=j$.
	
	Another consequence of $r_i^j{\bf B_r}r_k^l$  is that for all $\ba_{-ij}$ and $\bb_{-kl}$ that are permutations of each other, $r_i^j(\ba_{-ij})=r_k^l(\bb_{-kl})$.

	Now we examine the case in which there is just one permutation under which the roles are the same:
	\begin{definition}\label{twistedr}
		$r_i^j{\bf T_r}r_k^l$: the roles $r_i^j$ and $r_k^l$ are \textit{twistedly related} if  there exists $\sigma$ such that $\sigma(k)=i$ and $\sigma(l)=j$, verifying that \eqref{roeqro} holds for all $\ba$. 
	\end{definition}
	
	Here the existing permutation $\sigma$ prescribes how to `twist' the elements in a profile $\ba$ in such a way that the function $r_i^j(\ba)$ on $A_i\times A_j$ is the same as the function $r_k^l(\ba\sigma)$ on $A_k\times A_l$, acting as a translation between the points of view of players $i$ and $k$.
	
	\begin{lemma}\label{Tr}
		The relation ${\bf T_r}$ is reflexive, symmetric and transitive.
	\end{lemma}
	
	\begin{proof} Reflexivity is immediate using the identity permutation.

		Symmetry: Suppose that $r_i^j{\bf T_r}r_k^l$, that is, that there exists $\sigma$ such that $\sigma(k)=i$ and $\sigma(l)=j$, and for all $\ba$, $\pi_i(\ba)=\pi_k(\ba\sigma)$. Then there exists a permutation, namely $\sigma^{-1}$, such that $\sigma^{-1}(i)=k$ and $\sigma^{-1}(j)=l$, and for all $\ba$, $\pi_k(\ba)=\pi_k(\ba\sigma^{-1}\sigma)=\pi_i(\ba\sigma^{-1})$, so $r_k^l{\bf T_r}r_i^j$.
		
		Transitivity: Suppose that $r_i^j{\bf T_r}r_k^l$ and $r_k^l{\bf T_r}r_s^t$. Then  there is a permutation $\sigma$  such that for every $\ba\in A^I$, $\sigma(k)=i$, $\sigma(l)=j$, and $\pi_i(\ba)=\pi_k(\ba\sigma)$.  At the same time, there is a $\tau$ such that for every $\bb\in A^I$, $\tau(s)=k$, $\tau(t)=l$, and $\pi_k(\bb)=\pi_s(\bb\tau)$.  Putting $\bb=\ba\sigma$ we get that for each $\ba\in A^I$ the permutation $\sigma\tau$ is such that $\sigma\tau(s)=i$, $\sigma\tau(t)=j$ and $\pi_i(\ba)=\pi_k(\ba\sigma)=\pi_k(\bb)=\pi_s(\bb\tau)=\pi_s(\ba\sigma\tau)$.
		
	\end{proof}
	
	Finally, we consider the case in which the existing permutation depends on each profile:
	
	\begin{definition}\label{simulatesr}
		$r_i^j{\bf M_r}r_k^l$: the role $r_i^j$ \textit{simulates} role $r_k^l$ if  for each $\ba\in A$ there exists $\sigma_\ba$ such that $\sigma_\ba(k)=i$,  $\sigma_\ba(l)=j$, and $[r_i^j(\ba_{-ij})](a_i,a_j)=[r_k^l(\ba_{-ij}\sigma_\ba)](a_i,a_j)$. If we write $\bb=\ba\sigma_\ba$, then the previous equation becomes  $[r_i^j(\ba_{-ij})](a_i,a_j)=[r_k^l(\bb_{-kl})](b_k,b_l)$, given that $b_k=a_{\sigma(k)}=a_i$ and $b_l=a_j$.
	\end{definition}
	
	Notice that here the roles $r_i^j$ and $r_k^l$ are not equal under a single permutation. Instead, there is a map that for each profile $\ba\in A$ yields a permutation $\sigma_\ba$ such that $\pi_i(\ba)=\pi_k(\ba\sigma_{\ba})$. We can think of this map as establishing a relation of `simulation' between the roles $r_i^j$  and $r_k^l$ in a different way for each particular profile.
	
	It is clear that considering the permutation $\sigma_\ba$ to be the identity for every profile $\ba$, we can show that the relation $\bf M_r$ is reflexive. On the other hand, ${\bf M_r}$ ensures that the range of $r_i^j$ becomes a subset of that of $r_k^l$, but the converse is not true:
	
	\begin{example} \label{exSymM}
		To see that the relation ${\bf M_r}$ is not symmetric consider a game  with four players and  set of actions $\set{a,b,c,d}$ in which all the payoffs for all players are zero except for $\pi_2(a,b,c,d)=1$. We check that $r_1^2{\bf M_r}r_2^1$. The only permutations such that $\sigma(1)=2$ and $\sigma(2)=1$ are $(12)$ and $(12)(34)$. Since the only profiles that could yield $(a,b,c,d)$ under these permutations are  $(b,a,c,d)$ and $(b,a,d,c)$ it will be enough to show an appropriate $\sigma$ for each of those two profiles. We assign $(12)(34)$ to $(b,a,c,d)$ and $(12)$ to $(b,a,d,c)$ so  that
		$\pi_1(b,a,c,d)=\pi_2((b,a,c,d)(12)(34))=\pi_2(a,b,d,c)=0$ and   $\pi_1(b,a,d,c)=\pi_2((b,a,d,c)(12))=\pi_2(a,b,d,c)=0$.
		
		On the other hand, $r_2^1{\bf M_r}r_1^2$ is not the case, since $(a,b,c,d)(12)(34)=(b,a,d,c)$ and $(a,b,c,d)(12)=(b,a,c,d)$, so for both permutations $\pi_2(a,b,c,d)=1\neq\pi_1((a,b,c,d)\sigma)=0$.

	\end{example}

	\begin{lemma}\label{roles-transitivity} The  relation  ${\bf M_r}$ is transitive.
	\end{lemma}
	
	\begin{proof}
		Suppose that $r_i^j{\bf M_r}r_k^l$ and $r_k^l{\bf M_r}r_s^t$. Then for each $\ba\in A$ there is a permutation $\sigma_\ba$ such that $\sigma_\ba(k)=i$, $\sigma_\ba(l)=j$, and $\pi_i(\ba)=\pi_k(\ba\sigma_\ba)$, while for each $\bb\in A$ there is a $\tau_\bb$ such that  $\tau_\bb(s)=k$, $\tau_\bb(t)=l$, and $\pi_k(\bb)=\pi_s(\bb\tau_\bb)$.  Putting $\bb=\ba\sigma_\ba$ we get that for each $\ba\in A$ the permutation $\sigma_\ba\tau_\bb$ is such that $\sigma_\ba\tau_\bb(s)=i$, $\sigma_\ba\tau_\bb(t)=j$ and $\pi_i(\ba)=\pi_k(\ba\sigma_\ba)=\pi_k(\bb)=\pi_s(\bb\tau_\bb)=\pi_s(\ba\sigma_\ba\tau_\bb)$.
	\end{proof}

	\begin{lemma}\label{roles-inclusion}
		The relations defined above are ordered by inclusion: $\bf B_r\subseteq T_r\subseteq M_r$.
	\end{lemma}
	
	\begin{proof}
		Suppose that $r_i^j{\bf B_r}r_k^l$. Then any single  permutation $\sigma$ such that $\sigma(l)=j$ and $\sigma(k)=i$ is enough to show that $r_i^j{\bf T_r}r_k^l$.
		
		If we now assume $r_i^j{\bf T_r}r_k^l$, then letting $\sigma_\ba$ be the permutation $\sigma$ in the definition of  ${\bf T_r}$ for all $\ba\in A$, we prove that  $r_i^j\bMr r_k^l$.
	\end{proof}

	\section{Relations among players}
	
	The previous section presented the notion of roles and some relations among them, but now we want to shift the focus to relations among players. We will define relations between players $i$ and $j$ based on the relations their roles $r_i^j$ and $r_j^i$ have to each other. So, building on the concepts of blindly related, twistedly related and simulation between roles, we will define the notions of blindly related, twistedly related and simulating \textit{players}, along with a new one of rigidly related players.
	This allows us to describe some of  the possible situations of interdependence of players $i$ and  $j$ in the game.

	\begin{definition}\label{blind2}
		Two players $i,j$ are {\em blindly related} in a game $G$, (denoted by $i\mathbf{B}j$) if $r_i^j{\bf B_r}r_j^i$. 
	\end{definition}
	
	In other words,  the roles $r_i^j$ and $r_j^i$ of players $i$ and $j$ are  the same  under {\it any} permutation that exchanges their places. For every permutation $\sigma$ such that $\sigma(i)=j$ and $\sigma(j)=i$ and for every profile $\ba$,   $\pi_i(\ba)=\pi_j(\ba\sigma)$.
	
	\begin{lemma}\label{blind-transitivity2}
		In a game with at least four different players, $i, j, k$, and $l$, if $i{\bf B}j$ and $j{\bf B}k$ hold, then so does $i{\bf B}k$.
	\end{lemma}
	
	\begin{proof} The hypotheses $i{\bf B}j$ and $j{\bf B}k$ imply that for every profile $\ba$ and every permutation $\sigma$ such that $\sigma(i)=j$ and $\sigma(j)=i$, $\pi_i(\ba)=\pi_j(\ba\sigma)$, and for every profile $\bb$ and every permutation $\tau$ such that $\tau(j)=k$ and $\tau(k)=j$, $\pi_j(\bb)=\pi_k(\bb\tau)$.

		Given a profile $\ba$ and a permutation $\mu$ such that  $\mu(i)=k$ and $\mu(k)=i$, let us consider two cases. In the first one, we assume  $\mu(j)=l\neq j$. Now we can factor $ \mu=\sigma\tau$, where  $\sigma=(l\ k)(i\ j)$ and $\tau=(i\ j)(l\ k)\mu$ satisfy that $\sigma(i)=j$, $\sigma(j)=i$, $\tau(j)=k$, and $\tau(k)=j$. Then 
		
		\[\pi_i(\ba)=\pi_j(\ba\sigma)=\pi_k(\ba\sigma\tau)=\pi_k(\ba\mu).\]
		
		so $i{\bf B}k$ holds.
		
		In the case in which  $\mu(j)=j$, we  cannot get such a straightforward factorization. We write $\mu=\gamma (i \ k)$ where $\gamma$ is a permutation that leaves $i, j$ and $k$ invariant and consider the permutations: $\tau_1=(j\ k)(i\ l)$, $\sigma_1=(i\ j)$, $\sigma_2=(i\ j)(k\ l)$, $\tau_2=(j\ k)$, and $\sigma_3=\gamma(i\ j)$. We can check that $\sigma_s(i)=j$, $\sigma_s(j)=i$, $\tau_t(j)=k$, and $\tau_t(k)=j$ for $s=1,2,3$ and $t=1,2$. 
		
		We  compute the composition:  \[\sigma_3\tau_1\tau_2\sigma_2\sigma_1\tau_1= \gamma (i\ j)(j\ k)(i\ l)(j\ k)(i\ j)(k\ l)(i\ j)(j\ k)(i\ l)=\] \[\gamma (i\ j)(i\ l)(k\ l)(j\ k)(i\ l)=\gamma (i\ k)=\mu.\]
		
		Using the hypothesis, and the fact that by Lemma \ref{blindST} the relation $\bf B_r$ is symmetric, we get  for any $\ba\in A^I$: $\pi_i(\ba)=\pi_j(\ba\sigma_3)=\pi_k(\ba\sigma_3\tau_1)=\pi_j(\ba\sigma_3\tau_1\tau_2)=\pi_i(\ba\sigma_3\tau_1\tau_2\sigma_2)=\pi_j(\ba\sigma_3\tau_1\tau_2\sigma_2\sigma_1)=\pi_k(\ba\sigma_3\tau_1\tau_2\sigma_2\sigma_1\tau_1)=\pi_k(\ba\mu)$.
		
	\end{proof}
	
	\begin{example}\label{notrans}
		The transitivity of the relation $\bf B$ does not hold  if $n=3$. To see this, consider a game $G$ with three players and action set $\set{a,b,c}$. All the payoffs other than the ones indicated below are zero.
		\begin{center} $G$ \ \ \ \
			\begin{tabular}{c|c|c|c|c|c|c}
				profile &  (a,b,c) & (a,c,b)& (b,a,c)& (b,c,a)&(c,a,b)&(c,b,a)\\
				\hline
				payoffs  & (0,1,2) &(3,2,1) & (1,0,4)&(5,4,0)&(2,3,5)&(4,5,3)
			\end{tabular}
		\end{center}
		Here $1{\bf B}2$ and $2{\bf B}3$ hold, but  $1{\bf B}3$ does not, since $\pi_1((a,b,c))=0$ and $\pi_3((a,b,c) (1\ 3))=\pi_3((c,b,a))=3$. This example also shows that $\bf B$ is not reflexive, since $\pi_1((a,b,c))=0$ and $\pi_1((a,b,c)(2\ 3))=\pi_1((a,c,b))=3$ so $1{\bf B}1$ does not hold.
	\end{example}
	
	From the definition of $\bB$ and the symmetry of $\bBr$ (Lemma \ref{blindST}), it follows that $\bB$ is symmetric.
	
	Using the relation $\bB$ of blindly related players, we can present new characterizations of the notions of symmetric and anonymous games.

	\begin{lemma}
		A game is {\em anonymous} if and only if for every player $i\in I$, $i\bB i$ holds.
	\end{lemma}
	\begin{proof}
		Assume that a game is anonymous, so for every $i\in I$ and profiles $\ba$ and $\bb$, if $a_i= b_i$ and $\#\ba_{-i}=\#\bb_{-i}$ then $\pi_i(\ba)=\pi_i(\bb)$. To prove that $i\bB i$ we need to prove that if $\sigma$ is such that $\sigma(i)=i$ then $\pi_i(\ba)=\pi_i(\ba\sigma)$. It is enough to consider any such $\sigma$ and define $\bb=\ba\sigma$. Then $a_i= b_i$ and $\#\ba_{-i}=\#\bb_{-i}$ so $\pi_i(\ba)=\pi_i(\bb)=\pi_i(\ba\sigma)$ holds.
		
		If we assume that $i\bB i$ and $\ba, \bb$ are profiles such that $a_i= b_i$ and $\#\ba_{-i}=\#\bb_{-i}$ then there exists a permutation $\sigma$ such that $\sigma(i)=i$ and $\bb=\ba\sigma$, so $\pi_i(\ba)=\pi_i(\ba\sigma)=\pi_i(\bb)$. 
	\end{proof}

	\begin{lemma}
		A game is {\em symmetric} if and only if for every player $i, j\in I$, $i\bB j$ holds.
	\end{lemma}
	\begin{proof}
		The proof is similar to the previous one. In one direction,  if $a_i= b_j$ and $\#\ba_{-i}=\#\bb_{-j}$ and $\sigma$ is such that $\sigma(i)=j$ and $\sigma(j)=i$, define $\bb=\ba\sigma$. Then $a_i= b_j$ and $\#\ba_{-i}=\#\bb_{-j}$ so $\pi_i(\ba)=\pi_j(\bb)=\pi_i(\ba\sigma)$ holds.
		
		In the other direction, if $i\bB j$ and $\ba, \bb$ are profiles such that $a_i=b_j$ and $\#\ba_{-i}=\#\bb_{-j}$ then there exists a permutation $\sigma$ such that $\sigma(i)=j$, $\sigma(j)=i$, and $\bb=\ba\sigma$, so $\pi_i(\ba)=\pi_i(\ba\sigma)=\pi_i(\bb)$. 
	\end{proof}

	\begin{definition}\label{twisted2}
		Two players $i,j$ are {\em twistedly related} (denoted by $i\mathbf{T}j$)  if $r_i^j{\bf T_r}r_j^i$. 
	\end{definition}
	
	Players  $i$ and $j$ are twistedly related if there exists at least one permutation under which the roles $r_i^j$ and $r_j^i$ are the same. That is, there exists a  permutation $\sigma$ such that $\sigma(i)=j$ and $\sigma(j)=i$ and for every profile $\ba$,   $\pi_i(\ba)=\pi_j(\ba\sigma)$.
	
	It follows easily from the corresponding properties for the relation $\bTr$ from Lemma \ref{Tr} that $\bT$ is reflexive and symmetric. In general, it is not transitive as  the following example shows:
	
	\begin{example}\label{Tnotrans}
		We define a game with four players and set of actions $A=\set{a,b,c,d}$. We assume that the payoff for all players in all the profiles not indicated below are zero, so that the equations involving those profiles are trivially satisfied. 
		
		\begin{center} 
			\begin{tabular}{c|c|c|c|c|c|c}
				profile &  (a,b,c,d) & (a,c,b,d)& (c,a,b,d)&(c,b,a,d) &(b,c,a,d)&(b,a,c,d)\\
				\hline
				payoffs  & (1,2,3,0) &(4,3,2,0) & (3,4,5,0)&(6,5,4,0)&(5,6,1,0)&(2,1,6,0)
			\end{tabular}
		\end{center}
		
		In this game  $1{\bf T}2$ and $2\bT 3$ are satisfied, but $1{\bf T}3$ is not:  the permutations that could realize the relation are $(13)$ and $(13)(24)$, but we have that  $\pi_1((a,b,c,d)=1\neq\pi_3((a,b,c,d)(13))=\pi_3((c,b,a,d))=4$ and $\pi_1((a,b,c,d))=1\neq\pi_3((a,b,c,d)(13)(24))=\pi_3((c,d,a,b))=0$. 
		
	\end{example}
	
	\begin{definition}\label{simulation2}
		Player  $i$ {\em simulates} the situation of player $j$ ($i\mathbf{M}j$) if $r_i^j{\bf M_r}r_j^i$.
	\end{definition}
	
	When this is the case, for every profile $\ba$, there exists a permutation $\sigma_\ba$ such that $\sigma_\ba(i)=j$ and $\sigma_\ba(j)=i$ and
	\begin{equation}\label{simulation2-eq}
	\pi_i(\ba)=\pi_j(\ba\sigma_{\ba})
	\end{equation}
	
	The relation $\bM$ is reflexive, as follows easily from $\bMr$ being reflexive.
	
	\begin{lemma}\label{symmetryM}
		For games with 2 or 3 players, the relation $\bM$ is symmetric, but this condition fails for games with 4 or more players.
	\end{lemma}
	
	\begin{proof}
		For 2 or 3 players, if $i\neq j$ and $i\bM j$, then there is a single permutation, $\sigma=(ij)$ such that $\sigma(i)=j$ and $\sigma(j)=i$, so for every profile $\ba$, $\sigma_\ba=(ij)$. So, if for every profile $\ba$, $\pi_i(\ba)=\pi_j(\ba(ij))$, then we also get that for every profile $\bb$, $\pi_j(\bb)=\pi_j(\bb(ij)(ij))=\pi_i(\bb(ij))$ so $j\bM i$.
		
		To see that the relation ${\bf M}$ is not symmetric if there are four or more players,  we consider the game from example \ref{exSymM}. There we have that $1\bM 2$ holds but $2\bM 1$ does not.
	\end{proof}
	
	The relation $\bM$ is not transitive either, as can be seen by noticing that Example \ref{Tnotrans} applies to $\bM$ as well.
	
	To these conditions obtained from the corresponding ones for roles, we add:

	\begin{definition}\label{rigid2}
		Two players $i,j$ are  {\em rigidly related } (denoted with $i\mathbf{R}j$) if  for every profile $\ba$,
		\begin{equation}\label{eq-rigid2}
		\pi_i(\ba)=\pi_j(\ba(ij)).
		\end{equation}
	\end{definition}
	
	This can be interpreted as saying that the places of players $i$ and $j$  in the game matrix are interchangeable while the actions of the others remain fixed. This relation is clearly reflexive, using the identity permutation, and symmetric: for all $\ba \in A^I$, $\pi_j(\ba)=\pi_j(\ba(ij)(ij))=\pi_i(\ba(ij))$. Example \ref{notrans} shows also that the relation $\bR$ is not transitive, since in a game with three players, the only permutation that one need to use to verify the relation $i\bB j$ is $(ij)$, so in that example $1\bR 2$ and $2\bR 3$ hold, but $1\bR 3$ does not. Given that with the relation $\bR$, only permutations of the form $(ij)$ are used, adding more players does not change the fact that $\bR$ is not transitive.
	
	While this relation is weaker than $\bB$, when it holds for every pair of players, it is enough to guarantee the symmetry of the game.
	
	\begin{lemma}
		If in a game $G$, $i\mathbf{R}j$ is satisfied for all $i,j\in I$, then $i\mathbf{B}j$ is also satisfied for all $i, j \in I$. Therefore, the game is symmetric.
	\end{lemma}
	
	\begin{proof}
		Let $i$ and $j$ be fixed elements of $I$, and $\sigma$ a permutation such that $\sigma(i)=j$ and $\sigma(j)=i$. We can write $\sigma$ as the composition of disjoint cycles, one of which is the transposition $(ij)$ (see \cite{hungerford74algebra}). Furthermore, we assume for the moment that $\sigma=\tau(ij)$ where $\tau$ is a cycle $(x_1x_2\ldots x_k)$ disjoint from $(ij)$.
		
		Since $(x_1x_2\ldots x_k)=(ix_1)(x_1x_2)\ldots(x_ki)$, and using the condition $\mathbf{R1}$ accordingly in each step we get that
		
		\[\pi_i(\ba)=\pi_{x_1}(\ba(ix_1))=\pi_{x_2}(\ba(ix_1)(x_1x_2))=\ldots\]
		\[ =\pi_{x_k}(\ba(ix_1)(x_1x_2)\ldots(x_{k-1}x_k))=\pi_{i}(\ba(x_1x_2\ldots x_k))=\pi_j(\ba\sigma)\]
		
		It is clear that if $\sigma$ has more disjoint cycles, they can be dealt with in the same fashion.
		
	\end{proof}

	For a given game $G$, the inclusion between these relations is summarized as follows:
	
	\begin{proposition}\label{inclusions}
		$\bf B\subseteq R\subseteq T\subseteq M$
	\end{proposition}
	\begin{proof}
		Suppose that given $i,j \in I$, $i\mathbf{B}j$ is the case. Then every permutation $\sigma$ such that $\sigma(i)=j$ and $\sigma(j)=i$ satisfies $\pi_i(\ba)=\pi_j(\ba\sigma)$. In particular so does $\sigma=(ij)$. This expression is just equation~(\ref{eq-rigid2}), so $i\mathbf{R}j$ is satisfied.
		We have also proved that there exists one  permutation $\sigma$, namely $(ij)$, for which (\ref{eq-rigid2}) is satisfied, meaning that $i\mathbf{T}j$ is also the case.
		Finally, if $i\mathbf{T}j$ is the case, we can take for every profile $\ba$,  $\sigma_\ba = \sigma$  and we have:
		
		\[\pi_i(\ba)=\pi_j(\ba\sigma)=\pi_j(\ba\sigma_{\ba})\]
		
		\noindent meaning that $i\mathbf{M}j$ is the case.
	\end{proof}

	To see that these relations are different, consider the following examples:
	
	\begin{example}
		We define below some games with four players in which each one can choose either action $a$ or action $b$. We will assume that the payoff for all players in all the profiles not indicated below are zero (except for the case of $G'''$), so that the equations involving those profiles are trivially satisfied. A star will indicate that a payoff can take any value, and those values need not be all the same at all the occurrences of $*$. The players 1 and 2 are in different kinds of relations:
		
		\begin{center} $G$ \ \ \ \
			\begin{tabular}{c|c|c|c|c}
				profile &  (a,a,a,b) & (a,a,b,a)& (a,b,a,b)&(b,a,a,b) \\
				\hline
				payoffs  & (1,2,*,*) &(2,1,*,*) & (3,4,*,*)&(4,3,*,*)
			\end{tabular}
		\end{center}
		
		In the game $G$ above $1{\bf M}2$ is satisfied, but $1{\bf T}2$ is not, since the permutations that realize the relation are different for the profiles $(a,b,a,b)$ and $(a,a,a,b)$. More precisely, $\sigma_{(a,b,a,b)}=(12)$ and $\sigma_{(a,a,a,b)}=(12)(34)$.
		
		\begin{center}
			\begin{tabular}{c|c|c|c|c}
				&  (a,b,a,b) & (a,b,b,a)& (b,a,a,b)&(b,a,b,a) \\
				\hline
				$G'$ & (1,2,*,*) &(3,4,*,*) & (4,3,*,*)&(2,1,*,*)\\
				$G''$   & (1,2,*,*) &(3,4,*,*) & (2,1,*,*)&(4,3,*,*)\\
				
				$G'''$ & (1,2,*,*) &(1,2,*,*) & (2,1,*,*)&(2,1,*,*)\\
			\end{tabular}
		\end{center}
		
		In the game $G'$, $1{\bf T}2$ is satisfied, but not $1{\bf R}2$, since the permutation used is $(12)(34)$. In $G''$, $1{\bf R}2$ is satisfied, but $1{\bf B}2$ is not. If that were the case, using $\sigma=(12)(34)$ the equation $\pi_1(a,b,a,b)=\pi_2(b,a,b,a)$ would yield. Finally, the game $G'''$ illustrates that more identities need to be satisfied for $1{\bf B}2$ to hold, including some that involve profiles not in the table.
		
	\end{example}
	
	There are further ways in which the situations of two players may be regarded as similar or equivalent. We might want to compare not just the roles of $i$ and $j$ relative to each other, but also to find a way to match the role of each player relative to $i$ with one of the roles relative to $j$. So for $\bf X$ equal to each of $\bf B$, $\bf T$ and $\bf M$ we define two new relations. In the first one, the matching is given by a permutation:
	
	\begin{definition}
		Players $i$ and $j$ are in  the relation $\bf P^X$ ($i{\bf P^X}j$) if there exists a permutation $\tau$ such that $\tau(i)=j$ and for all $k\in I$, $r_i^k{\bf X_r}r_j^{\tau(k)}$.
	\end{definition}
	
	Relaxing the condition on the matching we obtain:
	
	\begin{definition}
		Players $i$ and $j$ are in  the relation $\bf Q^X$ ($i{\bf Q^X}j$) if $r_i^i{\bf X_r}r_j^j$ and for all $k\neq i$ there is an $l\neq j$ such that $r_i^k{\bf X_r}r_j^l$.
	\end{definition}
	
	It is clear that $\bf P^B\subseteq P^T\subseteq P^M$ and  $\bf Q^B\subseteq Q^T\subseteq Q^M$, while $\bf P^X\subseteq Q^X$ for $\bf X$ equal to $\bf B,T$ and $\bf M$. Inspecting the definitions more carefully reveals that for example $i{\bf P^B}j$ if and only if  there exists $\sigma$ such that $\sigma(j)=i$ and for all $\ba\in A$, $\pi_i(\ba)=\pi_j(\ba\sigma))$. It follows that $\bT\subseteq \bf P^B$.
	
	Similarly, $i{\bf Q^B}j$ means that for all $\ba\in A$ and $\sigma$ such that $\sigma(j)=i$, $\pi_i(\ba)=\pi_j(\ba\sigma))$. The rest of the definitions seem rather cumbersome and don't seem to have intuitive appeal.

	\section{Discussion}
	
	In this paper we presented different definitions of symmetric games from the literature, framing them under a combinatorial view to clarify their relations to each other.  This approach, given in terms of the group action of the group of permutations of the names of the players over the set of all strategic profiles, allows us the definition of symmetric game as invariance under these permutations. Such a definition can be easily found in a reinterpretation of the definition in \cite{nash51noncooperative}. 
	
	The definition of the role a player plays from the point of view of another lets us study further structure in a game using permutations. Roles can be compared to one another, and we have defined some different ways of doing so. To compare two players directly, we consider the roles each one of them plays with respect to the other. The contingent relations between roles and between players are diverse and present a rich behavior, as summarized in the following table.
	
	\begin{center}
		\begin{tabular}{|c||c|c|c|}
			\hline
			Relation	&  Reflexive & Symmetric& Transitive \\
			\hline\hline
			$\bBr$&{\scriptsize NO}&{\scriptsize YES} &{\scriptsize YES, if $n\ge 4$}\\
			& &{\scriptsize Lemma \ref{blindST}}&{\scriptsize Lemma \ref{blindST} }\\ \hline
			$\bTr$&{\scriptsize YES}&{\scriptsize YES}&{\scriptsize YES}\\
			&{\scriptsize Lemma \ref{Tr} }&{\scriptsize Lemma \ref{Tr}}&{\scriptsize Lemma \ref{Tr}}\\ \hline
			$\bMr$&{\scriptsize YES}&{\scriptsize NO}&{\scriptsize YES}\\
			& &{\scriptsize Example \ref{exSymM}}&{\scriptsize Lemma \ref{roles-transitivity}}\\ \hline
			$\bB$&{\scriptsize NO}&{\scriptsize YES}&{\scriptsize YES, if $n\ge 4$}\\
			&{\scriptsize Example \ref{notrans}} &  &{\scriptsize Lemma \ref{blind-transitivity2} and Example \ref{notrans}}\\ \hline
			$\bR$&{\scriptsize YES}&{\scriptsize YES}&{\scriptsize NO}\\
			& & &{\scriptsize Example \ref{notrans}}\\ \hline
			$\bT$&{\scriptsize YES}&{\scriptsize YES}&{\scriptsize NO}\\
			&  & &{\scriptsize Example \ref{Tnotrans}}\\ \hline
			$\bM$&{\scriptsize YES}&{\scriptsize NO, if $n\ge 4$}&{\scriptsize NO}\\
			& &{\scriptsize Lemma \ref{symmetryM}}&{\scriptsize Example \ref{Tnotrans}}\\ \hline
			
		\end{tabular}
	\end{center}
	
	Notice that of all the relations presented, the only one that is an equivalence relation is $\bf T_r$. One natural question is whether taking a quotient on the set of roles may lead to a simplification in the study of a game.
	
	Although we give our motivations in terms of how one player may find her situation reflected in that of other players, the equations that define the different relations are objective and can be checked against the payoff matrix of a strategic game. This may prove to be an important tool to analyze strategic games, finding players in similar situations in a natural and automatic way.
	
	The definitions we have given emphasize the assessment that a player may make of  another, which is a natural way in which humans analyze a game, but there is no difficulty in defining the role that a set of players plays with respect to another set of players. Thus, if $X, Y\subseteq I$, we may consider the role $r_X^Y:A_{-XY}\to\RR^{A_X\times A_Y}$, where $A_X=\prod_{i\in X}A_i$, and $A_{-XY}=\prod_{i\notin X\cup Y}A_i$, in an similar way to what we did for two players.

	A topic that we have not treated in this paper is that of the symmetries among {\em actions}. Nash proposes in \cite{nash51noncooperative} permutations over the set of actions in the game, and only then proceeds to study particular cases of those permutations, namely those that preserve the relation among players and actions that can be actually played by them. We avoided this complication just by assuming the set of actions is the same for all players. But this is an oversimplification. For instance, in the Battle of the Sexes the set of actions is the same for both players, but only by interchanging their names the symmetry that is apparent in the setting of the game can be formally justified. In \cite{ham18symmetry}, the case in which actions that have different names for different players can be identified is considered, so the Battle of the Sexes can be regarded as symmetrical. This label-independent approach leads to a combinatorial classification of games according to the symmetries found in them.   A further treatment of this extension of the concept of symmetry and a discussion of how it can contribute in finding equilibria of games can be found in \cite{Senci2018} or, in a quite different context \cite{Goranko17}.

	
	\section*{Acknowledgments}
	We thank Alfredo \'Alzaga  and Rodrigo Iglesias for fruitful discussion on these topics.
	

\begin{thebibliography}{BMV17}
		
		\bibitem[BFH09]{brandt08symmetries}
		Felix Brandt, Felix Fischer, and Markus Holzer.
		\newblock Symmetries and the complexity of pure {N}ash equilibrium.
		\newblock {\em J. Comput. System Sci.}, 75(3):163--177, 2009.
		
		\bibitem[BMV17]{Bouyer17Nash}
		Patricia Bouyer, Nicolas Markey, and Steen Vester.
		\newblock Nash equilibria in symmetric graph games with partial observation.
		\newblock {\em Inform. and Comput.}, 254(part 2):238--258, 2017.
		
		\bibitem[Chi96]{chichilnisky96actions}
		Graciela Chichilnisky.
		\newblock Actions of symmetry groups.
		\newblock {\em Soc. Choice Welf.}, 13(3):357--364, 1996.
		
		\bibitem[DM86]{dasgupta86existence}
		Partha Dasgupta and Eric Maskin.
		\newblock The existence of equilibrium in discontinuous economic games, {I}:
		Theory.
		\newblock {\em The Review of Economic Studies}, 53(1):1--26, 1986.
		
		\bibitem[DP07]{daskalakis2007computing}
		Constantinos Daskalakis and Christos Papadimitriou.
		\newblock Computing equilibria in anonymous games.
		\newblock In {\em Foundations of Computer Science, 2007. FOCS'07. 48th Annual
			IEEE Symposium on}, pages 83--93. IEEE, 2007.
		
		\bibitem[GKR17]{Goranko17}
		Valentin Goranko, Antti Kuusisto, and Raine R\"{o}nnholm.
		\newblock Rational coordination with no communication or conventions.
		\newblock In {\em Logic, rationality, and interaction}, volume 10455 of {\em
			Lecture Notes in Comput. Sci.}, pages 33--48. Springer, Berlin, 2017.
		
		\bibitem[Ham18]{ham18symmetry}
		Nicholas Ham.
		\newblock Notions of symmetry for finite strategic-form games, 2018.
		\newblock {\tt https://arxiv.org/abs/1311.4766v4}.
		
		\bibitem[Hef17]{Hefti17equilibria}
		Andreas Hefti.
		\newblock Equilibria in symmetric games: theory and applications.
		\newblock {\em Theor. Econ.}, 12(3):979--1002, 2017.
		
		\bibitem[Hun80]{hungerford74algebra}
		Thomas~W. Hungerford.
		\newblock {\em Algebra}, volume~73 of {\em Graduate Texts in Mathematics}.
		\newblock Springer-Verlag, New York, 1980.
		\newblock Reprint of the 1974 original.
		
		\bibitem[Kel92]{kelly92abelian}
		Jerry~S. Kelly.
		\newblock Abelian symmetry groups in social choice.
		\newblock {\em Math. Social Sci.}, 25(1):15--25, 1992.
		
		\bibitem[Nas51]{nash51noncooperative}
		John Nash.
		\newblock Non-cooperative games.
		\newblock {\em Ann. of Math. (2)}, 54:286--295, 1951.
		
		\bibitem[Par66]{parikh66context}
		Rohit~J. Parikh.
		\newblock On context-free languages.
		\newblock {\em J. Assoc. Comput. Mach.}, 13:570--581, 1966.
		
		\bibitem[Ser99]{serizawa99strategy}
		Shigehiro Serizawa.
		\newblock Strategy-proof and symmetric social choice functions for public good
		economies.
		\newblock {\em Econometrica}, 67(1):pp. 121--145, 1999.
		
		\bibitem[ST18]{Senci2018}
		Maximiliano Senci and Fernando Tohm\'e.
		\newblock Coordination through similarity.
		\newblock Submitted, 2018.
		
		\bibitem[Ste11]{stein11exchangeable}
		Noah~D. Stein.
		\newblock Exchangeable equilibria, 2011.
		\newblock Doctoral Thesis, MIT.
		
	\end{thebibliography}

\end{document}